\documentclass[12pt, article]{amsart}
\usepackage{geometry} 
\usepackage{ae} 
\usepackage[T1]{fontenc}
\usepackage[cp1250]{inputenc}
\usepackage{amsmath}
\usepackage{amssymb, amsfonts,amscd,verbatim}
\usepackage{MnSymbol}
\usepackage{mathrsfs} 
\usepackage{xypic}

\usepackage[normalem]{ulem}
\usepackage{hyperref}
\usepackage{indentfirst}
\usepackage{latexsym}
\input xy
\xyoption{all}

\usepackage{amsmath}    

\theoremstyle{plain}
\newtheorem{Pocz}{Poczatek}[section]
\newtheorem{Proposition}[Pocz]{Proposition}

\newtheorem{Theorem}[Pocz]{Theorem}
\newtheorem{Corollary}[Pocz]{Corollary}

\newtheorem{Lemma}[Pocz]{Lemma}

\newtheorem{Question}[Pocz]{Question}

\newtheorem{Example}[Pocz]{Example}

\theoremstyle{definition}
\newtheorem{Definition}[Pocz]{Definition}

\theoremstyle{remark}

\DeclareMathOperator*{\diam}{diam}
\DeclareMathAlphabet\mathbfcal{OMS}{cmsy}{b}{n}

\def\diam{\mathrm{diam}}

\DeclareMathOperator*{\st}{st}
\numberwithin{equation}{section}

\author{Logan ~ Higginbotham}
\address{Campbell University, Buies Creek, NC, USA}
\email{lhigginbotham@campbell.edu}
\author{Kevin ~ Sinclair}
\address{Shenandoah University, Winchester, VA, USA}
\email{ksinclai@su.edu}

\title{Asymptotic Filtered Colimits}

\date{ \today
}

\begin{document}

\begin{abstract}
If one has a collection of large scale spaces $\{(X_s,\mathcal{LSS}_s)\}_{s\in S}$ with certain compatibility conditions one may define a large scale space on $X=\bigcup\limits_{s\in S}X_s$ in a way where every function on $X$ is large scale continuous if and only if the function restricted to every $X_s$ is large scale continuous. This large scale structure is called the asymptotic filtered colimit of $\{(X_s,\mathcal{LSS}_s)\}_{s\in S}$. In this paper, we explore a wide variety of coarse invariants that are preserved between $\{(X_s,\mathcal{LSS}_s)\}_{s\in S}$ and the asymptotic filtered colimit $(X,\mathcal{LSS})$. These invariants include finite asymptotic dimension, exactness, property A, and being coarsely embeddable into a separable Hilbert space. We also put forth some questions and show some examples of filtered colimits that give an insight into how to construct filtered colimits and what may not be preserved as well.
\end{abstract}

\maketitle

\section{Introduction}
The main focus of this paper is to introduce the notion  of asymptotic filtered colimits. We do this by deriving the definition of the asymptotic filtered colimit construction and then look at some examples of asymptotic filtered colimits. We then show multiple coarse invariants that asymptotic filtered colimits preserve while stating some questions along the way. Finally, we end this paper with something that asymptotic filtered colimits do not preserve. We will start by introducing definitions associated with families of subsets of a set $X$ in order to define large scale structures:

\begin{Definition}
	Let $\mathcal{U}$ be a family of subsets of a set $X$ and let $V$ be a subset of $X$. The \textbf{star} of $V$ against $\mathcal{U}$, denoted $\st(V,\mathcal{U})$, is the set 
	$$\bigcup\limits_{\substack{U\in \mathcal{U}\\{U\cap V\neq \varnothing}}} U$$ 
	If $\mathcal{V}$ is another family of subsets of $X$, then the family of subsets of $X$ $\left\{\st(V,\mathcal{U})|V\in \mathcal{V}\right\}$ is denoted $\st(\mathcal{V},\mathcal{U})$ for convenience.
\end{Definition}

\begin{Definition}
	Let $\mathcal{U},\mathcal{V}$ be families of subsets of a set $X$. We say $\mathcal{U}$ is a \textbf{refinement} of $\mathcal{V}$ provided for every $U\in\mathcal{U}$ there is a $V\in\mathcal{V}$ so that $U\subseteq V$. In this same situation, we also say that $\mathcal{V}$ \textbf{coarsens} $\mathcal{U}$. Refiniement is denoted as $\mathcal{U}\prec\mathcal{V}$.
\end{Definition}

It is sometimes needed that we need to consider covers of $X$ instead of collections of subsets of $X$. To distinguish families of subsets of $X$ from covers of $X$, we call covers of $X$ scales:

\begin{Definition}
	Given a set $X$, we say $\mathcal{U}$ is a \textbf{scale} of $X$ if $\mathcal{U}$ is a family of subsets of $X$ that covers $X$. If $\mathcal{U}$ is a collection of subsets of $X$, we can make $\mathcal{U}$ into a cover via constructing $\mathcal{U}'=\mathcal{U}\cup \left\{\{x\}\right\}_{x\in X}$. This extension is often called the \textbf{trivial extension} of $\mathcal{U}$.
\end{Definition}

The definition of large scale structues was given in by Dydak in \cite{Dydak}. This interpretation of coarse structures give coarse geometry a more topological flavor.

\begin{Definition}
	\cite{Dydak} Let $X$ be a set. A \textbf{large scale structure} on $X$ is a non-empty set of families of subsets of $X$ $\mathcal{LSS}$ so that the following conditions are satisfied:
	\begin{enumerate}
		\item If $\mathcal{U},\mathcal{V}$ are families of subsets of X with $\mathcal{V}\in\mathcal{LSS}$ and each element $U$ of $\mathcal{U}$ consisting of more than one point is contained in some $V$ of $\mathcal{V}$, then $\mathcal{V}\in\mathcal{LSS}$.
		\item If $\mathcal{U},\mathcal{V}\in\mathcal{LSS},$ then $\st(\mathcal{U},\mathcal{V})\in\mathcal{LSS}$.
	\end{enumerate} 
	Elements $\mathcal{U}$ of $\mathcal{LSS}$ are called \textbf{uniformly bounded families} or \textbf{uniformly bounded scales}.
\end{Definition} 

We note here closure under refinements implies the first condition above. The advantage of having a weaker first requirement is that a large scale structure as defined "disregards" one point sets. That is, one point sets do not "change" the large scale structure. Also, the first item in the definition gives us that the cover $\{\{x\}\}_{x\in X}$ is uniformly bounded for any large scale structure.  We will now remind the reader of some preliminary definitions about maps from one large scale structure to another.

\begin{Definition}
	Let $(X,\mathcal{LSS}_X)$ and $(Y,\mathcal{LSS}_Y)$ be large scale spaces and let $f:X\to Y$. We say $f$ is \textbf{large scale continuous} or \textbf{bornologous} if for every\\ $\mathcal{U}\in\mathcal{LSS}_X,~f(\mathcal{U})\in\mathcal{LSS}_Y$, where $f(\mathcal{U})=\{f(U)|~U\in\mathcal{U}\}$. 
\end{Definition}



\begin{Definition}
	Let $(X,\mathcal{LSS}_X)$ and $(Y,\mathcal{LSS}_Y)$ be large scale spaces and let $f,g:X\to Y$. We say $f$ and $g$ are \textbf{close} provided there is a $\mathcal{V}\in\mathcal{LSS}_Y$ so that for any $x\in X,~f(x),g(x)\in V$ for some $V\in\mathcal{V}$.
\end{Definition}


\begin{Definition}
	Let $(X,\mathcal{LSS}_X)$ and $(Y,\mathcal{LSS}_Y)$ be large scale spaces and let $f:X\to Y$ be large scale continuous. f is a \textbf{coarse equivalence} if and only if there exists a large scale continuous map $g:Y\to X$ so that $f\circ g$ is close to $id_Y$ and $g\circ f$ is close to $id_X$.
\end{Definition}

\begin{Definition}
	Let $(X,\mathcal{LSS}_X)$ be a large scale space. A property $P$ is a \textbf{coarse invariant} if for any $(Y,\mathcal{LSS}_Y)$ that has property $P$ and is coarsely equivalent to $(X,\mathcal{LSS}_X)$ we have that $(X,\mathcal{LSS}_X)$ also has property $P$.
\end{Definition}

Coarse invariants include, but are not limited to: Metrizability, finite asymptotic dimension, asymptotic property C, property A, exactness, coarse amenability, and coarse embeddability into a separable Hilbert space. We shall explore some of these coarse invariants and show that they are preserved by the asymptotic filtered colimit construction. But we first must define it.

\section{Asymptotic Filtered Colimits}

We begin this chapter by introducing the notion of an asymptotic filtered colimit.

\begin{Definition}\label{afc}
	Suppose $X$ is a set with $\left\{(X_s,\mathcal{LSS}_s)\right\}_{s\in S}$ subsets of $X$ and for each $s\in S$, $X_s$ has the large scale structure $\mathcal{LSS}_s$.
	Further, assume $\bigcup\limits_{s\in S} X_s=X$ and for every $r,s\in S$ we have that the restrictions of the large scale structures $\mathcal{LSS}_r$ and $\mathcal{LSS}_s$ to the set $X_r\cap X_s$ coincide.
	Also, $\forall r,s\in S~\exists t\in S$ such that $X_r\cup X_s \subseteq X_t$.
	Then the \textbf{asymptotic filtered colimit} of $\left\{(X_s,\mathcal{LSS}_s)\right\}_{s\in S}$ of $X$ is the following large scale structure:\\
	$\mathcal{U}$ is uniformly bounded if and only if $\exists s\in S$ and $\mathcal{V}\in \mathcal{LSS}_s$ so that for any $U\in \mathcal{U}$ with $|U|>1~\exists V\in \mathcal{V}$ so that $U\subseteq V$ (and consequently $U\subseteq X_s$). The construction of creating $(X,\mathcal{LSS})$ from $\{(X_s,\mathcal{LSS}_s)\}$ is also called the \textbf{asymptotic filtered colimit construction}. 
\end{Definition}

We note here that another way to think of the uniformly bounded families in the asymptotic filtered colimit $\mathcal{LSS}$ of $\left\{(X_s,\mathcal{LSS}_s)\right\}_{s\in S}$ is the following: For any $\mathcal{U}\in\mathcal{LSS}$ there is an $s\in S$ so that $\mathcal{U}^*\in\mathcal{LSS}_s$, where $\mathcal{U}^*$ is $\mathcal{U}$ with all one-point sets outside of $X_s$ removed. As a consequence, for every $s\in S$, $\mathcal{LSS}_s\subseteq\mathcal{LSS}$. We will make use of these remarks moving forward.

\begin{Definition}
Suppose $X$ is a set and $\mathcal{LSS}$ is the asymptotic filtered colimit of subsets $\left\{(X_s, \mathcal{LSS}_s)\right\}_{s\in S}$ of $X$; let $\mathcal{U}\in\mathcal{LSS}$. Define $\mathbfcal{U^*}$ to be $\mathcal{U}$ with all one point sets outside of $X_s$ removed, where $X_s$ is the subset from $\{X_s\}_{s\in S}$ for which all elements of $\mathcal{U}$ of cardinality greater than one are a subset of by the definition of asymptotic filtered colimit.
\end{Definition}

\begin{Proposition}\label{afclss}
	The asymptotic filtered colimit of $\left\{(X_s,\mathcal{LSS}_s)\right\}_{s\in S}$ of $X$ (denoted $\mathcal{LSS}$) is indeed a large scale structure.
\end{Proposition}

\begin{proof}
	Let $\mathcal{U}\in \mathcal{LSS}$ and suppose we have a family of subsets of $X$, $\mathcal{W}$, so that $|W|>1$ implies there exists a $U\in \mathcal{U}$ so that $W\subseteq U$.
	Since $\mathcal{U}\in \mathcal{LSS}$, $\exists s\in S$ and $\mathcal{V}\in \mathcal{LSS}_s$ so that $|U|>1$ implies there is a $V\in \mathcal{V}$ such that $U\subseteq V$.
	If $|W|>1$ and $W\subseteq U$ along with $U\subseteq V$, then we have that $W\subseteq V$. Then by definition and choice of $s\in S$, we have $\mathcal{W}\in \mathcal{LSS}$.\\
	Now suppose $\mathcal{U},\mathcal{V}\in \mathcal{LSS}$. Then $\exists r\in S$ and $\mathcal{F}\in \mathcal{LSS}_r$ so that for any $U\in \mathcal{U}$ with $|U|>1$, we have $\exists F\in \mathcal{F}$ such that $U\subseteq F$. 
	Also, $\exists s\in S$ and $\mathcal{G}\in \mathcal{LSS}_s$ so that for any $V\in \mathcal{V}$ with $|V|>1$, we have $\exists G\in \mathcal{G}$ such that $V\subseteq G$.
	Select $t\in S$ such that $X_r\cup X_s\subseteq X_t$. We show that $\st(\mathcal{U},\mathcal{V})\in \mathcal{LSS}$. 
	Define $\mathcal{U^*}=\mathcal{U}\setminus\left\{U\in \mathcal{U}~|~U=\{x\},~x\in X\setminus X_r\right\}$. Likewise, define $\mathcal{V^*}=\mathcal{V}\setminus\left\{V\in \mathcal{V}~|~V=\{x\},~x\in X\setminus X_s\right\}$.
	Notice that we have $\mathcal{U^*}\in \mathcal{LSS}_r$ and that $\mathcal{V^*}\in \mathcal{LSS}_s$. Since $X_r\subseteq X_t$ and $X_s\subseteq X_t$ and the restrictions of the large scale structures of $\mathcal{LSS}_r$ and $\mathcal{LSS}_t$ (respectively $\mathcal{LSS}_s$ and $\mathcal{LSS}_t$) to the intersection $X_r\cap X_t=X_r$ (respectively $X_s\cap X_t=X_s$) coincide, we therefore have that $\mathcal{U^*}\in \mathcal{LSS}_r$ implies $\mathcal{U^*}\in \mathcal{LSS}_t$ along with $\mathcal{V^*}\in \mathcal{LSS}_s$ implies $\mathcal{V^*}\in \mathcal{LSS}_t$.
	Since $(X_r, \mathcal{LSS}_r)$ and $(X_s,\mathcal{LSS}_s)$ coincide with $(X_t,\mathcal{LSS}_t)$, we have that $\mathcal{U^*}\in\mathcal{LSS}_r$ implies there is a uniformly bounded family $\mathcal{U'}\in \mathcal{LSS}_t$ so that $\mathcal{U'}|_{X_r}=\mathcal{U^*}$, where $\mathcal{U'}|_{X_r}:=\left\{U'\cap X_r~|~U'\in\mathcal{U'}\right\}$. But this means that for every $U\in \mathcal{U}$ with $|U|>1$, there is a $U'\in \mathcal{U'}$ so that $U\subseteq U'$. Thus, $\mathcal{U^*}\in \mathcal{LSS}_t$.
	Since $\mathcal{U^*},\mathcal{V^*}\in \mathcal{LSS}_t$, we have that $\st(\mathcal{U^*},\mathcal{V^*})\in \mathcal{LSS}_t$. Let $\mathcal{W}=\st(\mathcal{U^*},\mathcal{V^*})\cup\left\{V\in \mathcal{V}~|~V=\{x\},~x\in X\right\}$. Then $\mathcal{W}\in \mathcal{LSS}$.
	We show that for any $U\in \mathcal{U}$ with $|\st(U,\mathcal{V})|>1$, we have that there exists $W\in \mathcal{W}$  so that $\st(U,\mathcal{V})\in \mathcal{W}$. This would show that $\st(\mathcal{U},\mathcal{V})\in \mathcal{LSS}$. 
	If $|U|>1$, then $U\in \mathcal{U^*}$ which implies that $\st(U,\mathcal{V})\in \mathcal{W}$.
	If $|U|=1$, then since $|\st(U,\mathcal{V})|>1$ we have that there is a $V\in \mathcal{V}$ such that $|V|>1$ and $U\subseteq V$. This gives us that $\st(U,\mathcal{V})\in \mathcal{W}$.
\end{proof}

Now that we've established that asymptotic filtered colimits are large scale structures, we will now provide a couple of examples of asymptotic filtered colimits:

\begin{Example}
Let $X$ be the group of all sequences with integer entries that converge to zero; the operation is componentwise addition. One might consider making $X$ a metric space by using the metric $d((x_1,x_2,...),(y_1,y_2,...))=\mathop{\Sigma}\limits_{i=1}^{\infty} |x_i-y_i|$. As a consequence, one yields a large scale structure $\mathcal{LSS}$ for $X$ which is induced from the metric. By that we mean $\mathcal{U}\in\mathcal{LSS}$ if and only if $\mathop{\sup}\limits_{U\in\mathcal{U}} \diam{(U)} <\infty$. This large scale structure is an asymptotic filtered colimit in the following way: Let $X_s=\mathbb{Z}^s\times\{0\}\times\{0\}\times...$ and let $\mathcal{LSS}_{s}$ be induced from the metric $d((x_1,...,x_s,0,0,...),(y_1,...,y_s,0,0,..))=\mathop{\Sigma}\limits_{i=1}^{s} |x_i-y_i|$. Then $(X,\mathcal{LSS})$ is the asymptotic filtered colimit of $\{(X_s,\mathcal{LSS}_s)\} _{s\in\mathbb{N}}$.
\end{Example}

\begin{Example}
Let $\{M_s, d^{M}_{s}\}_{s\in S}$ be a collection of metric spaces with $M_s\cap M_t=\varnothing$, $s\neq t$. Then one may define an $\infty$ metric on $X=\bigcup\limits_{s\in S} M_s$ in the following way: $d(x,y)=d^{M}_s(x,y)$ if $x,y\in M_s$ and $d(x,y)=\infty$ if $x\in M_s$ and $y\in M_t$, $s\neq t$. Let $\mathcal{LSS}$ be the large scale structure induced from the infinity metric $d$. Then $(X,\mathcal{LSS})$ is an asymptotic filtered colimit of $\{X_F, \mathcal{LSS}_F\}_{F\in \mathcal{P}_{fin}(S)}$, where $\mathcal{P}_{fin}(S)$ is the collection of all finite subsets of $S$, $X_F=\bigcup\limits_{s\in F} M_s$, and $\mathcal{LSS}_F$ is the large scale structure induced from the $\infty$ metric $d_F(x,y)=d^{M}_s(x,y)$ if $x,y\in M_s$ and $d_F(x,y)=\infty$ if $x\in M_s$, $y\in M_t$, $s\neq t$.
\end{Example}

The second example is a useful one to keep in mind when dealing with the asymptotic filtered colimit $\mathcal{LSS}$ of $\{X_s, \mathcal{LSS}_s\}_{s\in S}$; points within the same $X_s$ behave with respect to $\mathcal{LSS}_s$, while two points with one in $X_s$ and one outside of $X_s$ in certain circumstances may be regarded as very far away with respect to $\mathcal{LSS}$. Asymptotic filtered colimits can also be formed by "building up to $X$ from smaller $X_s$'s. This was shown in example 1.
The following proposition shows that large scale continuous functions of the asymptotic filtered colimit of $\{(X_s,\mathcal{LSS}_s)\}$ are precisely functions that are large scale continuous on every restriction to $(X_s,\mathcal{LSS}_s)$.

\begin{Proposition}\label{afcborn}
	Suppose $X$ is a set and $\mathcal{LSS}_X$ is the asymptotic filtered colimit
	of subsets $\left\{(X_s, \mathcal{LSS}_s)\right\}_{s\in S}$ of $X$ and $f:X\to Y$ is a function 
	to a large scale space $Y$. $f$ is bornologous if and only if $f|_{X_s}$ is bornologous for each $s$.
\end{Proposition}
	
\begin{proof}
	$\left( \Rightarrow\right):$ Let $s\in S$ and $\mathcal{U}_s\in \mathcal{LSS}_s$. Then notice that $\mathcal{U}_s\in \mathcal{LSS}_X$ which implies that $f(\mathcal{U}_s)\in \mathcal{LSS}_Y$. Since $U_s\in\mathcal{U}_s$ gives us $U_s\subseteq X_s$, we have $f(\mathcal{U}_s)=f|_{X_s}(\mathcal{U}_s)$.\\
	$\left( \Leftarrow \right):$ Let $\mathcal{U}\in \mathcal{LSS}_X$. Then there is an $s\in S$ and a $\mathcal{V}\in \mathcal{LSS}_s$ such that for any $U\in \mathcal{U}$ with $|U|>1$, there is a $V\in \mathcal{V}$ such that $U\subseteq V$.
	Define $\mathcal{U^*}=\mathcal{U}\setminus\left\{U\in \mathcal{U}~|~U=\{x\},~x\in X\setminus X_s\right\}$. Then $\mathcal{U^*}\in\mathcal{LSS}_s$ and $f(\mathcal{U^*})=f|_{X_s}(\mathcal{U^*})$. So $f(\mathcal{U^*})\in \mathcal{LSS}_Y$.
	We show that if $f(U)\in f(\mathcal{U})$ with $|f(U)|>1$, then $f(U)\in f(\mathcal{U^*})$. Indeed, $|f(U)|>1$ implies $|U|>1$ and hence $U\in\mathcal{U^*}$ which implies $f(U)\in f(\mathcal{U^*})$. So $f(\mathcal{U})\in \mathcal{LSS}_Y$.
\end{proof}

It turns out that slowly oscillating functions behave similarly to large scale continuous functions with respect to asymptotic filtered colimits. The following definitions are slight generalizations of the ones found in \cite{Weighill}

\begin{Definition}
	Let $\left(X,\mathcal{LSS}\right)$ be given and let $\mathcal{U}\in \mathcal{LSS}$.
	We say a $\mathbfcal{U}$\textbf{-chain component} of $X$ is an equivalence class of the following equivalence relation. $x\sim y$ if and only if there is a finite sequence $\left\{U_i \right\}_{i=1}^{n} \subseteq \mathcal{U}$ such that $U_i\cap U_{i+1} \neq \varnothing$ for every $i$ and $x\in U_1$ along with $y\in U_n$.
	A \textbf{coarse chain component} of $x\in X$ is the union of its $\mathcal{U}$-chain components, where $\mathcal{U}$ ranges over every uniformly bounded family of $\mathcal{LSS}$. 
	A subset $B\subseteq X$ is called \textbf{weakly bounded} if its intersection with each coarse chain component is contained in some $U$ for $U\in \mathcal{U}$ and $\mathcal{U}\in \mathcal{LSS}$.
\end{Definition}

\begin{Definition}
	Let $f:X\to Y$ where $\left(X,\mathcal{LSS}\right)$ is a large scale structure and $Y$ is a metric space. $f$ is \textbf{slowly oscillating} if $\forall \mathcal{U}\in \mathcal{LSS}$ and $\epsilon >0~\exists B\subseteq X$ weakly bounded such that for any $U\in \mathcal{U}$ with $U\not\subseteq B$ implies $\diam(f(U))<\epsilon$.
\end{Definition}

\begin{Proposition}\label{afcso}
	Let $X$ be a set and let $\mathcal{LSS}$ be the asymptotic filtered colimit of $\left\{\left(X_s, \mathcal{LSS}_s\right)\right\}_{s\in S}$. Let $Y$ be a metric space and let $f:X\to Y$. Then $f$ is slowly oscillating if and only if $f|_{X_s}$ is slowly oscillating for all $s\in S$.
\end{Proposition}

\begin{proof}
	$\left( \Rightarrow\right):$ Let $\mathcal{U}_s\in\mathcal{LSS}_s$ and $\epsilon >0$. Then there is a $B\subseteq X$ weakly bounded such that for any $\mathcal{U}_s\in \mathcal{LSS}_s$ with $U_s \not\subseteq B$ implies $\diam(f(U_s))<\epsilon$.
	But $U_s \subseteq X_s$ implies $f(U_s)=f|_{X_s}(U_s)$ and we are done with choice of weakly bounded subset $B\cap X_s$. 
	Indeed, suppose $U_s\in \mathcal{U}_s$ and $U\not\subseteq \left(B\cap X_s \right)$. Then since $U_s \subseteq X_s$, we have that $U_s \not\subseteq B$ which implies $\diam(f(U_s))=\diam(f|_{X_s}(U_s))<\epsilon$.\\
	$\left( \Leftarrow\right):$ Let $\mathcal{U}\in \mathcal{LSS}$ and $\epsilon >0$. Then there is an $s\in S$ and $\mathcal{V}\in \mathcal{LSS}_s$ such that for any $U\in \mathcal{U}$ with $|U|>1$ implies $U\subseteq V$ for $V\in \mathcal{V}$.
	Define $\mathcal{U^*}$ to be $\mathcal{U}$ with one point sets removed outside of $X_s$. Then $\mathcal{U^*}\in \mathcal{LSS}_s$ which implies there is a $B\subseteq X_s \subseteq X$ weakly bounded such that for any $U\in \mathcal{U^*}$ with $U\not\subset B$ we have that $\diam(f(U))<\epsilon$. 
	Notice that for any $U\in \mathcal{U}\setminus\mathcal{U}^*$ with $U\not\subseteq B$, we have that $\diam(f(U))=0<\epsilon$. Therefore, $B$ is a choice of a weakly bounded set with the property that for any $U\in \mathcal{U}$ with $U\not\subset B$, we have $\diam(f(U))<\epsilon$. So $f$ is slowly oscillating.
\end{proof}

We now will showcase various coarse properties that are preserved by the asymptotic filtered colimit We will begin with metrizability of coarse spaces. For completeness, we remind the reader of the following from \cite{Dydak}. In particular, this statement is a combination of proposition 1.6 and theorem 1.8 in the paper cited:

\begin{Proposition}\label{metlss}
	Let $\mathcal{LSS}$ be a large scale structure on a set $X$ and suppose there exists a set of families of $X$, $\mathcal{LSS}'$, such that for any $\mathcal{B}_1,\mathcal{B}_2\in \mathcal{LSS}'$
	there exists $\mathcal{B}_3\in \mathcal{LSS}'$ such that $\mathcal{B}_1\cup \mathcal{B}_2\cup \st \left(\mathcal{B}_1,\mathcal{B}_2 \right)$ refines $\mathcal{B}_3$.
	Then if the cardinality of $\mathcal{LSS}'$ is countable, then $\mathcal{LSS}$ is metrizable as a coarse space.
\end{Proposition}

\begin{proof}
	See \cite{Dydak}.
\end{proof}

\begin{Proposition}\label{afcmet}
	Let $\left(X,\mathcal{LSS}\right)$ be an asymptotic filtered colimit of $\left\{\left(X_s,\mathcal{LSS}_s\right)\right\}_{s\in S}$ and that for every $s\in S$ we have that $X_s$ is metrizable as a coarse space. Then if $S$ is countable, then $X$ is metrizable as a coarse space.
\end{Proposition}

\begin{proof}
	By \ref{metlss} we have that for every $s\in S$, there is a $\mathcal{LSS'}_s$ such that $|\mathcal{LSS'}_s|$ is countable and $\forall \mathcal{B}_1^s, \mathcal{B}_2^s \in \mathcal{LSS'}_s~\exists \mathcal{B}_3^s$ such that $\mathcal{B}_1^s \cup \mathcal{B}_2^s \cup \st(\mathcal{B}_1^s , \mathcal{B}_2^s)$ is a refinement of $\mathcal{B}_3^s$.
	Let $\mathcal{LSS}'=\bigcup\limits_{s\in S} \mathcal{LSS'}_s$. Then $|\mathcal{LSS}'|$ is countable since the countable union of countable sets is countable.
	Let $\mathcal{A'}_s,\mathcal{B'}_r \in \mathcal{LSS'}$. Then note that there is a $t\in S$ so that $X_r\cup X_s \subseteq X_t$ and $\mathcal{A'}_s , \mathcal{B'}_s \in \mathcal{LSS'}_t$ which implies there is a $\mathcal{W'}_t\in \mathcal{LSS'}_t$ so that $\mathcal{A'}_s \cup \mathcal{B'}_r \cup \st(\mathcal{A'}_s, \mathcal{B'}_r)\in \mathcal{W'}_t$
	Since, $\mathcal{W'}_t\in \mathcal{LSS'}$, we have by \ref{metlss} that $X$ is metrizable as a coarse space.
\end{proof}

We use the following definition of Asymptotic Dimension from \cite{Dydak}:

\begin{Definition}
	Let $\left(X,\mathcal{LSS}\right)$ be a large scale structure. We say $\left(X,\mathcal{LSS}\right)$ has \textbf{asymptotic dimension at most n} if every uniformly bounded family $\mathcal{U}$ in $X$ there is a uniformly bounded coarsening $\mathcal{V}$ such that the multiplicity of $\mathcal{V}$ is at most $n+1$ (i.e. each point $x\in X$ is contained in at most $n+1$ elements of $\mathcal{V}$).
\end{Definition}

\begin{Proposition}\label{afcfad}
	Suppose $X$ is a set and $\mathcal{LSS}$ is the asymptotic filtered colimit
	of subsets $\left\{(X_s, \mathcal{LSS}_s)\right\}_{s\in S}$ of $X$
	The asymptotic dimension of $X$ is at most $n$
	if and only if the asymptotic dimension of every $(X_s,\mathcal{LSS}_s)$ is at most $n$.
\end{Proposition}

\begin{proof}
	$\left(\Rightarrow\right):$ Let $\mathcal{U}_s\in \mathcal{LSS}_s$.
	Then we have that $\mathcal{U}_s\in \mathcal{LSS}$ and hence $\mathcal{U}_s$ has a coarsening $\mathcal{V}$ with multiplicity at most $n+1$. The desired coarsening is $\mathcal{V'}=\left\{V\cap X_s~|~V\in \mathcal{V}\right\}$ \\
	$\left( \Leftarrow \right):$ Let $\mathcal{U}\in \mathcal{LSS}$. 
	Then there is an $s\in S$ and a $\mathcal{V}\in \mathcal{LSS}_s$ such that for any $U\in \mathcal{U}$ with $|U|>1$, there is a $V\in \mathcal{V}$ such that $U\subseteq V$. Define $\mathcal{U^*}$ as before.
	Then $\mathcal{U^*}\in \mathcal{LSS}_s$ and hence there is a coarsening $\mathcal{W}\in \mathcal{LSS}_s$ with multiplicity at most $n+1$.
	Then the family $\mathcal{W}\cup\left\{U\in \mathcal{U}~|~U=\{x\},~x\in X\setminus X_s\right\}$. is the desired coarsening of $\mathcal{U}$ with multiplicity at most $n+1$.
\end{proof}

Given how nicely asymptotic filtered colimits preserve finite asymptotic dimension, one might wonder if the asymptotic filtered colimit construction preserves asymptotic property C. It is not currently known if this is the case. Below is a definition of asymptotic property C that agrees with the more commonly seen definition for metric spaces. We note here that this generalized definition is preserved under subspaces and is also a coarse invariant:

\begin{Definition}
	Let $(X,\mathcal{LSS})$ be given. We say that $(X,\mathcal{LSS})$ has \textbf{asymptotic property C} or {APC} if for any sequence of uniformly bounded families $\mathcal{U}_1 \prec \mathcal{U}_2 \prec...$ there is a natural number $n$ and $\mathcal{V}_1,...,\mathcal{V}_n\in\mathcal{LSS}$ so that $\bigcup\limits_{i=1}^{n} \mathcal{V}_i$ covers X and for all $j$, $1\leq j\leq n$, $V,V'\in \mathcal{V}_j$ with $V\neq V'$, $\st(V,\mathcal{U}_j)\cap V'=\varnothing$.
\end{Definition}

From the remarks, we get a simple corollary.

\begin{Corollary}
	Let $X$ be a set and let $\mathcal{LSS}$ be the asymptotic filtered colimit of $\left\{\left(X_s, \mathcal{LSS}_s\right)\right\}_{s\in S}$. If $\mathcal{LSS}$ has APC, then $\left(X_s, \mathcal{LSS}_s\right)$ has APC for any $s\in S$.
\end{Corollary}

\begin{Question}
	Let $X$ be a set and let $\mathcal{LSS}$ be the asymptotic filtered colimit of $\left\{\left(X_s, \mathcal{LSS}_s\right)\right\}_{s\in S}$. If for all $s\in S,~ (X_s,\mathcal{LSS}_s)$ has APC, then does $(X,\mathcal{LSS})$ have APC?
\end{Question}

\begin{Question}
	Suppose $X$ is a set and $\mathcal{LSS}$ is the asymptotic filtered colimit of subsets $\left\{(X_s, \mathcal{LSS}_s)\right\}_{s\in S}$ of $X$ and that the asymptotic dimension of each $(X_s,\mathcal{LSS}_s)$ is finite. Does $(X,\mathcal{LSS})$ have asymptotic property C?
\end{Question}

We'll now show that exactness is preserved by the asymptotic filtered colimit construction. We remind the reader of the following definitions. The following is adapted from \cite{Guentner}:

\begin{Definition}
	Let $X$ be a set. We say $\left(f_i\right)_{i\in I}$ is a \textbf{partition of unity} of $X$ if $f_i:X\to [ 0,\infty )$ for all $i$ and for all $x\in X$, $\sum\limits_{i\in I} f_i(x)=1$.
\end{Definition}

The following definition is adapted from \cite{Guentner}:

\begin{Definition}
	Let $(X,\mathcal{LSS})$ be a large scale structure. $(X,\mathcal{LSS})$ is \textbf{exact} if for every $\mathcal{U}\in\mathcal{LSS}$ and $\epsilon >0$ there exists a partition of unity $(f_i)_{i\in I}$ of $X$ so that
	the cover of $X$, $\mathcal{V}=\left\{\mathrm{support}(f_n)~|~i\in I\right\}$, is uniformly bounded and that if (for $U\in \mathcal{U}$) $x,y\in U$, then $\sum\limits_{i\in I} |f_i(x)-f_i(y)| <\epsilon$.
\end{Definition}

\begin{Proposition}\label{afcexact}
	Suppose $X$ is a set and $\mathcal{LSS}$ is the asymptotic filtered colimit of subsets $\left\{X_s\right\}_{s\in S}$ of $X$. $(X,\mathcal{LSS})$ is exact if and only if for each $s\in S$, $(X_s,\mathcal{LSS}_s)$ is exact.
\end{Proposition}

\begin{proof}
	$\left( \Rightarrow \right):$ Let $\mathcal{U}_s\in\mathcal{LSS}_s$ and $\epsilon >0$. Note that for any $s\in S$ $\mathcal{LSS}_s\subseteq\mathcal{LSS}$.  Then we have $\mathcal{U}_s\in \mathcal{LSS}$; since $(X,\mathcal{LSS})$ is exact, we can find the desired partition of unity of $X$. Restrict this partition of unity of $X$ to a partition of unity of $X_s$. This shows that $(X_s,\mathcal{LSS}_s)$ is exact.\\
	$\left( \Leftarrow \right):$ Let $\mathcal{U}\in \mathcal{LSS}$ and $\epsilon >0$. Then there exists $s\in S$ and $\mathcal{V}\in \mathcal{LSS}_s$ such that for every $U\in \mathcal{U}$ with $|U|>1$ there exists a $V\in \mathcal{V}$ so that $V\subseteq U$. 
	Let $\mathcal{U^*}$ be as shown in other proofs. Then $\mathcal{U^*}\in \mathcal{LSS}_s$ which means there is a partition of unity of $X_s$, $\left(f_i\right)_{i\in I}$ so that the family $\left\{\mathrm{support}(f_i)~|~i\in I\right\}$ is uniformly bounded and if $U\in \mathcal{U}$ and $x,y\in U$, then $\sum\limits_{i\in I} |f_i(x)-f_i(y)|<\epsilon$.
	For any value $j\in X\setminus X_s$, define $f_j:X\to [0,\infty)$ via $f_j(j)=1$ and zero elsewhere. Also, for any $i\in I$ extend $f_i:X_s \to [0,\infty)$ to $X$ by defining $f_i(j)=0$ for any $j\in X\setminus X_s$. Let the set $J$ index the various $f_j$'s and let $K=I\cup J$. 
	We claim that $\left(f_k\right)_{k\in K}$ is the desired partition of unity of $X$. Indeed, notice that aside from a collection of one point sets (i.e. $\mathrm{support}(f_j)$ for $j\in J$), we have that the family $\left\{\mathrm{support}(f_k)~|~k\in K\right\}=\left\{\mathrm{support}(f_i)~|~i\in I\right\}\in\mathcal{LSS}_s\subset \mathcal{LSS}$. 
	Now let $U\in \mathcal{U}$. If $|U|=1$, then we have that $x,y\in U$ implies that $x=y$ and thus $\sum\limits_{k\in K} |f_k(x)-f_k(y)|=0< \epsilon$. If $|U|>1$, then we have that $U\subseteq X_s$ and since $\left(f_i\right)_{i\in I}$ is a partition of unity for $X_s$ and that $f_j(U)\equiv 0$, we have that $x,y\in U$ implies $\sum\limits_{k\in K} |f_k(x)-f_k(y)|=\sum\limits_{i\in I} |f_i(x)-f_i(y)|<\epsilon$. Now we show that for every $x\in X$, $\sum\limits_{k\in K} f_k(x)=1$. Suppose $x\in X_s$. then for any $i\in I,~f_i(x)=0$ and there is a unique $j\in J$ so that $f_j(x)=1$. So $\sum\limits_{k\in K}f_k(x)=1$. If $x\in X\setminus X_s$, then we have that for any $j\in J,~f_j(x)=0$ and since $\left(f_i\right)_{i\in I}$ form a partition of unity for $X_s$, we have that $\sum\limits_{k\in K} f_k(x)=\sum\limits_{i\in I} f_i(x)=1$. 
\end{proof}

We will now show that embeddability into separable Hilbert spaces is preserved by the asymptotic filtered colimit construction. The notion of coarse embeddability was introduced in \cite{Yu}.  Recall that for any two separable Hilbert spaces $G$ and $H$, there is an isometric isomorphism between the two. We will also use some pinch space theory. The following definition and theorem is adapted from \cite{Holloway}:

\begin{Definition}
	Let $(X,\mathcal{LSS})$ be a large scale space, $K$ a metric space, and $c>0$. We say $(X,\mathcal{LSS})$ \textbf{c-pinch-spaces to K} if for every $\mathcal{U}\in\mathcal{LSS}$ and $\epsilon>0$ there is a $\mathcal{V}\in\mathcal{LSS}$ and a function $f:X\to K$ so that $\sup\limits_{U\in \mathcal{U}} \diam(f(U)) <\epsilon$ and for every $x,y\in X$ so that $\{x,y\}\not\subseteq V$ for every $V\in \mathcal{V}$ we have that $d_K(f(x),f(y))\geq c$.
\end{Definition}

\begin{Theorem}\label{pinch}
	If $X$ is a metric space, then $X$ coarsely embedds into a Hilbert space if and only if $X$ c-pinch-spaces to a Hilbert space for some $c>0$.
\end{Theorem}

\begin{Theorem}\label{afchilb}
	Let $S$ be a countable index set and let $H$ be a fixed separable Hilbert space. Let $\left( X,\mathcal{LSS}\right)$ be the asymptotic filtered colimit of $\left\{\left(X_s, \mathcal{LSS}_s\right)\right\}_{s\in S}$ with every $X_s$ countable. Then $\left( X,\mathcal{LSS}\right)$ coarsely embedds into $H$ if and only if $\left(X_s, \mathcal{LSS}_s\right)$ coarsely embedds into $H$ for all $s\in S$.
\end{Theorem}

\begin{proof}
	$\left( \Rightarrow \right):$ This follows via restriction of the embedding function $f:X\to H$ to any $X_s$.\\
	$\left( \Leftarrow \right):$ Note that $\bigoplus\limits_{s\in S} H\cong H$ since $S$ is countable. Likewise, $H\oplus H \cong H$. We show $\left(X, \mathcal{LSS}\right)$ 1-pinch-spaces to $H$.
	Let $\mathcal{U}\in \mathcal{LSS}$ and $\epsilon >0$. Define $\mathcal{U}^*$ to be $\mathcal{U}$ with one point sets removed. Then by definition of $\mathcal{LSS}$, $\mathcal{U}^*\in \mathcal{LSS}_s$ for some s.
	Since $\left(X_s, \mathcal{LSS}_s\right)$ 1-pinch-spaces to H, there exists $f_{\epsilon, s}^{\mathcal{U}^*}:X_s\to H$ and $\mathcal{W}_s \in \mathcal{LSS}_s$ such that $\sup\limits_{U\in\mathcal{U}^*} \diam({f_{\epsilon, s}^{\mathcal{U}^*}(U)}) <\epsilon$ 
	and for any $x,y\in X_s$ with $\{x,y\}\not\subseteq W$ for every $W\in \mathcal{W}_s$ we have $\|f_{\epsilon, s}^{\mathcal{U}^*}(x)-f_{\epsilon, s}^{\mathcal{U}^*}(y)\|\geq 1$ (the norm is in $H$).
	Now, since $X=\bigcup\limits_{s\in S} X_s$ and $X_s$ is countable for every $s$, we may index an orthonormal basis of $H$ via $\left\{e_x\right\}_{x\in X}$.
	Furthermore, define $f_{\epsilon}^{\mathcal{U}}:X\to H\oplus H$ via $f_{\epsilon}^{\mathcal{U}}(x)=\left(f_{s, \epsilon}^{\mathcal{U}^*}(x), 0\right)$ for any $x\in X_s$ and $f_{\epsilon}^{\mathcal{U}}(x)=\left( 0, e_x \right)$ for any $x$ not in $X_s$.
	Define $\mathcal{V}\in \mathcal{LSS}$ as $\mathcal{V}=\mathcal{W}_s \cup \left\{{x}\right\}_{x\in X}$. We will show that $\left(f_{\epsilon}^{\mathcal{U}}, \mathcal{V}\right)$ satisfies the 1-pinch space conditions.
	Note that $\epsilon > \sup\limits_{U\in\mathcal{U}^*} \diam(f_{\epsilon, s}^{\mathcal{U}^*}(U))~=~ \sup\limits_{U\in\mathcal{U}} \diam(f_{\epsilon}^{\mathcal{U}}(U))$ 
	since $|U|>1$ implies $U\subseteq X_s$ and $U\in \mathcal{U}^*$ which implies that $\diam(U)<\epsilon$. If $|U|=1$, then $\diam(f_{\epsilon}^{\mathcal{U}}(U))=0<\epsilon$. Hence, $\sup\limits_{U\in\mathcal{U}} \diam(f_{\epsilon}^{\mathcal{U}}(U))<\epsilon$.
	Now, let $x,y\in X$ so that $\left\{x,y\right\}\not\subseteq V$ for every $V\in \mathcal{V}$. We have three cases:\\
	Suppose $\left\{x,y\right\}\subseteq X\setminus X_s$. Then $f_{\epsilon}^{\mathcal{U}}(x)=(0,e_x)$ and $f_{\epsilon}^{\mathcal{U}}(y)=(0,e_y)$. Then we have that $\|(0,e_x)-(0,e_y)\|_{H\oplus H}=\sqrt{\|0\|_{H}^{2} + \|e_x -e_y\|_{H}^{2}}=\sqrt{2}>1$.\\
	Suppose $\left\{x,y\right\}\subseteq X_s$. Then $\left\{x,y\right\}\not\subseteq V$ for every $V\in \mathcal{V}$ implies that $\left\{x,y\right\}\not\subseteq W$ for every $W\in \mathcal{W}_s$.
	Then we have that $\|f_{\epsilon}^{\mathcal{U}}(x)-f_{\epsilon}^{\mathcal{U}}(y)\|_{H\oplus H}=\sqrt{\|f_{\epsilon, s}^{\mathcal{U}^*}(x)-f_{\epsilon, s}^{\mathcal{U}^*}(y)\|_{H}^{2}+ \|0\|_{H}^{2}}\geq 1$ by assumption that $\left(X_s, \mathcal{LSS}_s\right)$ 1-pinch-spaces to $H$.\\
	Suppose that $x\in X_s$ and $y\in X\setminus X_s$. Then $f_{\epsilon}^{\mathcal{U}}(x)=(f_{\epsilon, s}^{\mathcal{U}^*}(x), 0)$ and $f_{\epsilon}^{\mathcal{U}}(y)=(0,e_y)$.
	Then we have that $\|f_{\epsilon}^{\mathcal{U}}(x)-f_{\epsilon}^{\mathcal{U}}(y)\|_{H\oplus H}= \|(f_{\epsilon, s}^{\mathcal{U}^*}(x),0)-(0,e_y)\|_{H\oplus H}=\sqrt{\|f_{\epsilon, s}^{\mathcal{U}^*}(x)-0\|_{H}^{2} + \|0-e_y\|_{H}^{2}}\geq 1$.\\
	So in all cases, $\|f_{\epsilon}^{\mathcal{U}}(x)-f_{\epsilon}^{\mathcal{U}}(y)\|_{H\oplus H}\geq 1$. Defining $h:X\to H$ to be the composition of $f_{\epsilon}^{\mathcal{U}}$ with the isometric isomorphism from $H\oplus H$ to $H$, we have that $\left(h,\mathcal{V}\right)$ 1-pinch-spaces to $H$ which means that $X$ coarsely embedds into $H$.
\end{proof}

We will now show that coarse amenability is preserved through the asymptotic filtered colimit construction. This definition of coarse amenability is given in \cite{Cencelj}.

\begin{Definition}
	Let $X$ be a set, $A\subseteq X$, and $\mathcal{U}$ a family of subsets of $X$. Then the \textbf{horizon of A against} $\mathcal{U}$, denoted $hor(A,\mathcal{U})$, is the set $\left\{U\in \mathcal{U}| A\cap U\neq \varnothing\right\}$.
\end{Definition}

Here are some useful properties of the horizon that we will use:

\begin{Lemma}\label{hor}
	Let $X$ be a set, $A,B\subseteq X$, and $\mathcal{U}, \mathcal{V}$ be families of subsets of $X$. Then:
	\begin{enumerate}
		\item $A\subseteq B \Rightarrow hor(A,\mathcal{U})\subseteq hor(B, \mathcal{U})$
		\item $\mathcal{U}\prec\mathcal{V}\Rightarrow hor(A,\mathcal{U})\subseteq hor(A,\mathcal{V})$
		\item $A\subseteq B$ and $\mathcal{U}\prec\mathcal{V}\Rightarrow hor(A,\mathcal{U})\subseteq hor(B,\mathcal{V})$.
	\end{enumerate}
\end{Lemma}

\begin{proof}
	Let $U\in hor(A,\mathcal{U})$. Then $\varnothing\neq U\cap A \subseteq U\cap  B$ which implies that $B\cap U \neq \varnothing$. So $U\in hor(B,\mathcal{U})$.\\
	For the second item, let $U\in hor(A,\mathcal{U})$. Then $\varnothing\neq U\cap A$. But $U\in\mathcal{U}\prec\mathcal{V}$ implies $U\in hor(A,\mathcal{V})$.
	The last statement is a combination of the first two.
\end{proof}

\begin{Definition}
	Let $\left(X,\mathcal{LSS}\right)$ be a large scale structure. Then $\left(X,\mathcal{LSS}\right)$ is \textbf{coarsely amenable} if for every $\mathcal{U}\in \mathcal{LSS}$ and $\epsilon>0$, there exists 
	$\mathcal{V}\in\mathcal{LSS}$ so that for any $x\in \bigcup\limits_{U\in\mathcal{U}} U$, $|hor(\st({x},\mathcal{U}),\mathcal{V})|<\infty$ and 
	$$\frac{|hor({x},\mathcal{V})|}{|hor(\st({x},\mathcal{U}),\mathcal{V})|}>1-\epsilon$$
	For simplicity, we denote $hor(\{x\},\mathcal{V})$ as $hor(x,\mathcal{V})$ and $hor(\st(\{x\},\mathcal{U}),\mathcal{V})$ as $hor(\st(x,\mathcal{U}),\mathcal{V})$.
\end{Definition}

\begin{Theorem}\label{afcamen}
	Suppose $S$ is an index set, $\left( X,\mathcal{LSS}\right)$ the asymptotic filtered colimit of $\left\{\left(X_s, \mathcal{LSS}_s\right)\right\}_{s\in S}$. Then $\left( X,\mathcal{LSS}\right)$ is coarsely amenable if and only if $\left(X_s, \mathcal{LSS}_s\right)$ be coarsely amenable for every $s\in S$. 
\end{Theorem}

\begin{proof}
	$\left( \Rightarrow \right):$ It is shown in \cite{Cencelj} that coarse amenability is preserved by taking subspaces.\\
	$\left( \Leftarrow \right):$Let $\mathcal{U}\in \mathcal{LSS}$. Then for some $s\in S,~\mathcal{U}^*\in\mathcal{LSS}_s$, where $\mathcal{U}^*$ is $\mathcal{U}$ with one point sets outside of $X_s$ removed.
	As $\mathcal{LSS}_s$ is coarsely amenable, there is a $\mathcal{V}^*\in\mathcal{LSS}_s$ so that for any $x\in \bigcup\limits_{U\in\mathcal{U}^*} U$, $|hor(\st({x},\mathcal{U}^*),\mathcal{V}^*)|<\infty$ and 
	$\frac{|hor({x},\mathcal{V}^*)|}{|hor(\st({x},\mathcal{U}^*),\mathcal{V}^*)|}>1-\epsilon$. 
	Define $\mathcal{V}=\mathcal{V}^*~\cup \left(\mathcal{U}\setminus\mathcal{U}^*\right)$. Then $\mathcal{V}\in\mathcal{LSS}$. Note that by construction, $\mathcal{V}\setminus\mathcal{V}^*=\mathcal{U}\setminus\mathcal{U}^*$. We now show that for any $x\in\bigcup\limits_{U\in\mathcal{U}} U$, 
	$hor(\st(x,\mathcal{U}),\mathcal{V})=hor(\st(x,\mathcal{U}^*),\mathcal{V}^*)\cup hor(\st(x,\mathcal{U}\setminus \mathcal{U}^*),\mathcal{V}\setminus \mathcal{V}^*)$. Furthermore, we will show that the $hor(\st(x,\mathcal{U}^*),\mathcal{V}^*)\cap hor(\st(x,\mathcal{U}\setminus \mathcal{U}^*),\mathcal{V}\setminus \mathcal{V}^*)=\varnothing$.\\
	$\left(\subseteq\right):$ Let $V\in hor(\st(x,\mathcal{U}),\mathcal{V})$. Then $V\in \mathcal{V}^*$ or it isn't. Suppose $V\in \mathcal{V}^*$. Then there is a $U\in \mathcal{U}$ so that $x\in U$ and $U\cap V\neq\varnothing$. We will show that $U\in \mathcal{U}^*$.
	Suppose not (for contradiction). Then $U\subseteq \left(X\setminus X_s \right)$ and $|U|=1$. Hence $U=\left\{ x\right\}$ and $U\subseteq V$ as $U\cap V\neq\varnothing$. Thus, $x\in V$ so $V\not\subseteq X_s$ which implies $V\not\in\mathcal{V}^*$ which is a contradiction. So we must have that $U\in \mathcal{U}^*$ hence $V\in hor(\st(x,\mathcal{U}^*),\mathcal{V}^*)$.
	Now, if $V\not\in \mathcal{V}^*$, then there is a $U\in \mathcal{U}$ so that $x\in U$ and $U\cap V\neq\varnothing$. As $V\not\in\mathcal{V}^*$, we have that $|V|=1$ which means that $V\subseteq U$.
	As $V\not\subseteq X_s$, we have that $U\not\subseteq X_s$ which implies (by definition of $\mathcal{LSS}$) $|U|=1$. So $U=V=\left\{x\right\}$ and $U\in \mathcal{U}\setminus\mathcal{U}^*$.
	Therefore, $x\in U$ implies $U\in\st(x,\mathcal{U}\setminus\mathcal{U}^*)$ which implies $V\in hor(\st(x,\mathcal{U}\setminus \mathcal{U}^*),\mathcal{V}\setminus \mathcal{V}^*)$.\\
	$\left(\supseteq\right)$: This follows via two applications of the previous lemma.\\
	We now show that $hor(\st(x,\mathcal{U}^*),\mathcal{V}^*)\cap hor(\st(x,\mathcal{U}\setminus \mathcal{U}^*),\mathcal{V}\setminus \mathcal{V}^*)=\varnothing$.
	Note that $hor(\st(x,\mathcal{U}\setminus\mathcal{U}^*),\mathcal{V}\setminus \mathcal{V}^*)=\{x\}$ or is the empty set since $hor(\st(x,\mathcal{U}\setminus \mathcal{U}^*))=\{x\}$ or the empty set. If this set is the singelton $\{x\}$, then $x\not\in X_s$ which implies that $\st(x,\mathcal{U}^*)=\varnothing$
	which means that $hor(\st(x,\mathcal{U}^*),\mathcal{V}^*)\cap hor(\st(x,\mathcal{U}\setminus \mathcal{U}^*),\mathcal{V}\setminus \mathcal{V}^*)=\varnothing$ as desired.\\
	Since $hor(x,\mathcal{V})=hor(x,\mathcal{V}^*)\cup hor(x,\mathcal{V}\setminus\mathcal{V}^*)$ (and the union is disjoint) and by the previous lemma $hor(x,\mathcal{V}\setminus\mathcal{V}^*)=hor(\st(x,\mathcal{U}\setminus \mathcal{U}^*),\mathcal{V}\setminus \mathcal{V}^*)$, we therefore have that:
	
	$$\frac{|hor(x,\mathcal{V})|}{|hor(\st(x,\mathcal{U}),\mathcal{V})|}= \frac{|hor(x,\mathcal{V}^*)|+|hor(x,\mathcal{V}\setminus\mathcal{V}^*)|}{|hor(\st(x,\mathcal{U}^*),\mathcal{V}^*)| + |hor(\st(x,\mathcal{U}\setminus \mathcal{U}^*),\mathcal{V}\setminus \mathcal{V}^*)|}=$$\\
	$$\frac{|hor(x,\mathcal{V}^*)|+|hor(x,\mathcal{V}\setminus\mathcal{V}^*)|}{|hor(\st(x,\mathcal{U}^*),\mathcal{V}^*)| + |hor(x,\mathcal{V}\setminus\mathcal{V}^*)|}$$
	
	If we can show the fraction above is greater than $1-\epsilon$ for any $x\in\bigcup\limits_{U\in\mathcal{U}} U$, then we're done. Let $x\in\bigcup\limits_{U\in\mathcal{U}} U$. Then $x\in\bigcup\limits_{U\in \mathcal{U}^*} U$ or $x\in\bigcup\limits_{U\in\mathcal{U}\setminus\mathcal{U}^*} U$. 
	If $x\in\bigcup\limits_{U\in\mathcal{U}\setminus\mathcal{U}^*} U$, then $x\in X\setminus X_s$ and for some $U\in\mathcal{U}$, $U=\{x\}$. Thus, $|hor(x,\mathcal{V}\setminus\mathcal{V}^*)|=1$ and $|hor(\st(x,\mathcal{V}^*))|=0$ (as $x\not\in X_s$) so $|hor(\st(x,\mathcal{U}),\mathcal{V})|=1<\infty$ and for any $\epsilon$ between one and zero, $\frac{|hor(x,\mathcal{V})|}{|hor(\st(x,\mathcal{U}),\mathcal{V})|}=1>1-\epsilon$.
	If $x\in\bigcup\limits_{U\in\mathcal{U}^*} U$, then we have that $x\in X_s$ which implies that $|hor(x,\mathcal{V}\setminus\mathcal{V}^*)|=0$ and hence $\frac{|hor(x,\mathcal{V})|}{|hor(\st(x,\mathcal{U}),\mathcal{V})|}=\frac{|hor(x,\mathcal{V}^*)|}{|hor(\st({x},\mathcal{U}^*),\mathcal{V}^*)|}>1-\epsilon$. So $\left(X,\mathcal{LSS}\right)$ is coarsely amenable.
\end{proof}

We will show that property A is preserved by the asymptotic filtered colimit construction. The following definitions are from \cite{Kevin}. They are generalizations from the typical definition of property A (defined only on metric spaces with bounded geometry) to large scale spaces with bounded geometry:

\begin{Definition}
	$(X,\mathcal{LSS})$ is a \textbf{bounded geometry coarse space} if for any $\mathcal{U}\in\mathcal{LSS},~\sup\limits_{U\in\mathcal{U}} |U|~<\infty$. 
\end{Definition}

\begin{Definition}\label{Prop. A}
	Let $(X,\mathcal{LSS})$ be a bounded geometry coarse space. We say that $(X,\mathcal{LSS})$ has \textbf{property A} if for any $\epsilon>0$ and $\mathcal{U}\in\mathcal{LSS}$ there is a $\mathcal{V}\in\mathcal{LSS}$ and a family of subsets of $X\times\mathbb{N},~\left\{A_x\right\}_{x\in X}~,$ so that for each $x\in X$:
	$|A_x|<\infty$, $(x,1)\in A_x$, $A_x\subseteq\st(x,\mathcal{V})\times\mathbb{N}$, and for any $y\in\st(x,\mathcal{U})$ we have $\frac{|A_x\Delta A_y|}{|A_x\cap A_y|}<\epsilon$, where $A_x\Delta A_y$ is the symmetric difference of $A_x$ and $A_y$.
\end{Definition}

\begin{Proposition}\label{afcpropA}
	Let $(X,\mathcal{LSS})$ be an asymptotic filtered colimit of $\{(X_s,\mathcal{LSS}_s)\}_{s\in S}$. If $(X_s,\mathcal{LSS}_s)$ is a bounded geometry coarse space with property A for every $s\in S$, then $(X,\mathcal{LSS})$ is a bounded geometry coarse space with property A.
\end{Proposition}

\begin{proof}
	 Note that $(X,\mathcal{LSS})$ is a bounded geometry coarse space since for any $\mathcal{U}\in\mathcal{LSS}$, we have that $\mathcal{U}^*\in\mathcal{LSS}_s$ for some $s\in S$ and that $(X_s,\mathcal{LSS}_s)$ is a bounded geometry coarse space.
	We now show that $(X,\mathcal{LSS})$ has property A. Let $\mathcal{U}\in\mathcal{LSS}$ and $\epsilon >0$. Then we have that for some $s\in S,~\mathcal{U}^*\in\mathcal{LSS}_s$. Since $(X_s,\mathcal{LSS}_s)$ has property A, we have that there is a $\mathcal{V}_s\in\mathcal{LSS}_s$ and a collection of subsets of $X_s\times\mathbb{N}$, $\{A_x\}_{x\in X_s}$, so that the requirements of property A are satisfied in $(X_s,\mathcal{LSS}_s)$.
	Note that $\mathcal{V}_s\in\mathcal{LSS}$ and define $\mathcal{V}\in\mathcal{LSS}$ via $\mathcal{V}=\mathcal{V}_s~\cup\{\{x\}|x\in X\setminus X_s\}$. Define $\{B_x\}_{x\in X}$ via $B_x=A_x$ if $x\in X_s$ and $B_x=\{(x,1)\}$ otherwise. We show that $\mathcal{V}$ and $\{B_x\}_{x\in X}$ satisfy the requirements in the definition of property A.
	Let $x\in X$. Then $|B_x|<\infty$ and $(x,1)\in B_x$ are obvious. $B_x\subseteq \st(x,\mathcal{V})$ since $x\in X_s$ implies that $B_x=A_x\subseteq \st(x,\mathcal{V}_s)\times\mathbb{N}=\st(x,\mathcal{V})\times\mathbb{N}$. Otherwise, $B_x=(x,1)\subseteq \st(x,\mathcal{V})\times\mathbb{N}=\{x\}\times\mathbb{N}$.
	Lastly, let $y\in\st(x,\mathcal{U})$. If $x\in X_s$, then we have that $\st(x,\mathcal{U})=\st(x,\mathcal{U}^*)$ hence $y\in\st(x,\mathcal{U}^*)$  (i.e. $y\in X_s$) and $B_x=A_x$ and $B_y=A_y$. So $\frac{|A_x\Delta A_y|}{|A_x\cap A_y|}<\epsilon$ since $(X_s,\mathcal{LSS}_s)$ has property A. If $x\in X\setminus X_s$, then $y\in\st(x,\mathcal{U})$ implies that $y=x$. Hence $|B_x\Delta B_y|=0$ and $\frac{|B_x\Delta B_y|}{|B_x\cap B_y|}=0<\epsilon$. So $(X,\mathcal{LSS})$ has property A.
\end{proof}

The converse of this theorem is most likely true. One would need to show that Property A is preserved by subspaces. It was shown in \cite{NowakYu} that this is true in the case of uniformly discrete metric spaces. 

We have presented multiple properties that are preserved through asymptotic filtered colimits. It turns out that close functions are not preserved through asymptotic filtered colimits. The following is such an example:

\begin{Example}
Let $X=\left( 0,1\right]$ and let $X_n=\left[\frac{1}{n+1}, 1\right]$ for $n\in \left\{1,2,...\right\}$. Let $X_n$ have the large scale structure induced by the metric of absolute value. Then we have that $\bigcup\limits_{n=1}^{\infty} X_n=X$ and that $X_n\subseteq X_{n+1}$ for every $n$.
Let $\mathcal{LSS}$ be the asymptotic filtered colimit of $\left\{(X_n,\mathcal{LSS}_n)\right\}_{n\in\mathbb{N}}$ of $X$. Let $f:X\to\mathbb{R}$ be defined via $f(x)=\frac{1}{x}$. Also, define $g:X\to\mathbb{R}$ be defined via $g(x)=1$ and give $\mathbb{R}$ the large scale structure induced by the metric of absolute value.
For any $n$, we have that $X_n$ is a compact set. Since the function $|f-g|$ is continuous on $X_n$, we have that $f|_{X_n}$ is close to $g|_{X_n}$ for all $n$.
However, $f$ is not close to $g$. Indeed, suppose for contradiction that $f$ is close to $g$. Then there is a uniformly bounded family $\mathcal{V}$ of $\mathbb{R}$ so that for any $x\in X$, $\left\{f(x),g(x)\right\}\subseteq V$ for some $V\in \mathcal{V}\cup \left\{\left\{y\right\}~|~y\in Y\right\}$. By definition of the large scale structure of $\mathbb{R}$,  there exists an $M>0$ so that for any $V\in \mathcal{V},~\diam(V)<M$.
This implies that for any $x\in X,~|f(x)-g(x)|<M$ i.e. for any $x\in \left(0,1\right]$, $\frac{1-x}{x}<M$. This is a contradiction. Indeed, choose $x=\frac{1}{M+2}$.\\
\end{Example}




\end{document}